\documentclass[12pt]{amsart}
\usepackage{geometry} 
\usepackage{microtype} 

\usepackage{mathtools}
\usepackage{cite}
\makeatletter
\newcommand{\citenoadjust}[1]{{\let\cite@adjust\empty#1}}
\makeatother

\usepackage{graphicx}

\DeclareMathOperator{\Int}{int}

\newcommand{\Reals}{\mathbb{R}}
\newcommand{\ZZ}{\mathbb{Z}}
\newcommand{\Sphere}{\mathbb{S}}

\newtheorem{fact}{Fact}
\newtheorem{theorem}{Theorem}

\newtheorem{lemma}{Lemma}
\newtheorem{coro} {Corollary}

\title{Convex equipartitions of volume and surface area}
\author{Alfredo Hubard}
\author{Boris Aronov}
\address[A.~Hubard]{Courant Institute of Mathematical Sciences, NYU\\
         New York, NY 10012 USA}
\email[A.~Hubard]{hubard@cims.nyu.edu}
\address[B.~Aronov]{Department of Computer Science and Engineering\\
         Polytechnic Institute of NYU\\
         Brooklyn, NY 11201 USA}
\email[B.~Aronov]{aronov@poly.edu}

\begin{document}
\begin{abstract}
  We show that, for any prime $p$ and any convex body $K$ (i.e.,
  a compact convex set with interior) in $\Reals^d$, there exists a
  partition of $K$ into $p^k$ convex sets with equal volume and equal
  surface area. Similar results regarding equipartitions with respect to continuous functionals on convex bodies are also proved.
   \end{abstract}
\maketitle

\section{Introduction}
Imagine that you are cooking chicken at a party. You will cut the raw chicken fillet with a sharp knife, marinate each of the pieces in a spicy sauce and then fry the pieces. The surface of each piece will be crispy and spicy. Can you cut the chicken so that all your guests get the same amount of crispy crust and the same amount of chicken? Thinking of two-dimensional convex chickens, Nandakumar and Ramana Rao~\cite{nr} asked the ``interesting and annoyingly
resistant question''~\cite{bar} of whether \emph{a convex body in the
  plane can be partitioned into $n$ convex regions with equal areas
  and equal perimeters}.  This is easy for $n=2$ and known for $n=3$
\cite{bar}. We confirm this conjecture and its natural generalization to higher dimensions for $n$ a prime power.
 \begin{theorem} \label{main}
  Given a convex body $K$ in $\Reals^d$ and a prime $p$, it is
  possible to partition $K$ into $p^k$ convex bodies with equal
  $d$-dimensional volumes and equal $(d-1)$-dimensional surface area, for any integer $k\geq 1$.
\end{theorem} 
In fact, we derive this result from the following much more general one. Let $\mathcal{K}^d$ be the space of convex sets
in $\Reals^d$ with the Hausdorff metric. We say that a measure $\mu$ is \emph{nice} if, for every hyperplane $H$, we have $\mu(H)=0$ and $\mu$ is in the weak closure of the set of absolutely continuous measures with a second moment.
\begin{theorem}\label{extra}
Given a nice measure $\mu$ on $\Reals^d$, 
 a
convex body $K\in \mathcal{K}^d$, and a family of $d-1$ continuous
functionals $G_1,G_2, \ldots G_{d-1} \colon \mathcal{K}^d \to \Reals$,
a prime number $p$, and an integer $k \geq 1$,
there is a partition of $K$ into $p^k$ convex bodies $K_1, K_2 \ldots K_{p^k}$, so that  $$\mu(K_i)=\frac{\mu(K)}{p^k}$$ and $$G_j(K_1)=G_j(K_2)= \ldots =G_j(K_{p^k}),$$ for all $ 1\leq i\leq p^k$ and $1\leq j\leq d-1$.
\end{theorem} 
We use tools from equivariant topology and from the theory of optimal transport.

We remark at this point that in an earlier draft of this paper we had
the result only for primes smaller than the dimension. In the closing
remarks of that version we pointed out that the full conjecture
followed from a Borusk--Ulam-type statement which is very similar to
theorem~\ref{top} below.  Roman Karasev contacted us immediately after
we uploaded our paper \cite{ah} to the \emph{arXiv}. He informed us
that he had come up with a very similar (but not identical)
geometrical solution to the Nandakumar-Ramana Rao conjecture
~\cite{nr}  for prime powers; his paper \cite{k10} is now
available. He also informed us that theorem~\ref{top} below had been
proven \cite[lemma 6]{k09}; however, the proof is sketched in that paper, where the reader is referred to Vasil'ev's paper \cite{v} for some details, Vasil'ev himself assumes certain algebraic topology expertise.  For completeness we include a detailed write up of the Karasev-Vasil'ev-Fuks proof of theorem~\ref{top} that avoids the language of index theory and equivariant cohomology and should be accessible to anyone with a basic understanding of (co)homology and vector bundles. We explain the relationship to the Gromov-Borsuk-Ulam theorem of \cite{gro} and \cite{mem} and include proofs of a closely related result independently found by Karasev and Pablo Soberon \cite{s} confirming a conjecture of Kaneko and Kano \cite{kk} that generalizes the Ham Sandwich theorem of Banach.

\section{Optimal Transport}

Given two probability measures $\mu_X$ and $\mu_Y$ on spaces $X$ and
$Y$, respectively, a \emph{coupling} or a \emph{transportation plan}
is a measure $\Pi$ on the product space $X \times Y$ such that, for
measurable sets $A \subset X$ and $B
\subset Y$, the identities $\Pi(A \times Y)=\mu_X(A)$ and $\Pi(X \times B)=\mu_Y(B)$ hold. The Monge-Kantorovich
\emph{optimal transport problem} imposes a cost function $c \colon X
\times Y \to \Reals$ and asks for the transportation plan which
minimizes $\int_{X \times Y} c(x,y) d\Pi$.  Of particular interest is
the case when $X=Y=\Reals^d$, the cost function is the square of the
Euclidean norm, both $\mu_X$ and $\mu_Y$ have a second moment, and
$\mu_X$ is absolutely continuous with respect to Lebesgue measure. In
this case the optimal transport plan can be shown to be essentially
unique~\cite{b} and can be described by a map $T$ that pushes $\mu_X$
to $\mu_Y$, i.e., the measure is defined by $d\Pi=d \mu_x
\delta_{y=T(x)}$.  We denote this map by $T_{\mu_X}^{\mu_Y}$. The
function $d_W(\mu,\nu)=\int |x-T_\mu^\nu(x)|_2^2d\mu(x)$ is a
pseudo-metric on Borel probability measures with a second moment,
known as \emph{Wasserstein distance}, see~\cite{v08}.

\subsection{From absolutely continuous to absolutely discrete}
We will be interested in the case in which $\mu_X$ is an absolutely
continuous measure and $\mu_Y$ is supported on a finite set.  For each
$n$-tuple of pairwise distinct points (\emph{sites}) $x_1,x_2, \ldots
x_n \in \Reals^d$ with corresponding \emph{radii} $r_1,r_2, \ldots r_n
\in \Reals$ (called \emph{weights}\footnote[1]{The term ``weight'' is confusing in our context since our main objects are sums of delta masses.}
 in the literature), the \emph{power diagram} is a tessellation of $\Reals^d$ which generalizes the Dirichlet-Voronoi diagram. A
point $x\in \Reals^d$ gets assigned to the cell $C_i$ corresponding to
the site $x_i$ if $f_i(x)=|x-x_i|^2-r_i$ is minimal among all
$i$'s.\footnote[2]{The term ``power'' comes from Euclidean
  geometry. Recall that the power of a point $p$ with respect to a
  circle of radius $r$ and center $y$, that does not contain the point
  $p$, is $|p-y|^2-r^2$.  Returning to the choice of the word
  ``radius'', we choose not to write the square since our $r_i$'s
  could be negative.  Aurenhammer~\cite{a} prefers to keep the square
  and allow $r$ to be an arbitrary complex number.} 
Expanding the equation $f_i(x)=f_j(x)$, it is easy to see that
the cells are convex. Following similar notation to \cite{v08} we
denote by $P_2^{ac}(\Reals^d)$ the set of probability measures on
$\Reals^d$ that are absolutely continuous with respect to Lebesgue
measure and have a finite second moment.  By the Radon-Nikodym
theorem, for any such measure $\mu$, there exists an absolutely
continuous function $f_\mu \in L_1(\Reals^d, \mathcal{L}^d)$ such that
for every measurable set $A$ we have 
\[
  \mu(A)=\int_A  f_\mu d\mathcal{L}^d,
\]
where $\mathcal{L}^d$ denotes the $d$-dimensional Lebesgue measure. The function $f_{\mu}$ is usually referred to as the \emph{Radon-Nikodym derivative} of $\mu$ with respect to $\mathcal{L}^d$. We will denote it by $[\frac{d\mu}{d\mathcal{L}^d}]$. We will be interested in a subset of $P_2^{ac}(\Reals^d)$, namely the subset for which the Radon-Nikodym derivative is strictly positive. Denote this set of measures by $P_2^{{ac}^+}(\Reals^d)$.
\begin{theorem}[\citenoadjust{\cite{AHA,b,mac}}]\label{AHA}
  Let $\mu\in P_2^{{ac}^+}(\Reals^d)$  and $\nu$ be a convex combination of delta masses at
  $n$~ distinct points $\langle x_1,x_2, \ldots,
  x_n\rangle$. There exists a set of radii $\langle r_1, r_2, \ldots,
  r_n \rangle$ such that the power diagram $\langle C_1, C_2, \ldots,
  C_n\rangle$ defined by $\langle x_1,x_2, \ldots, x_n\rangle$ and
  $\langle r_1, r_2, \ldots, r_n \rangle$ gives the unique solution to the
  optimal transport problem, i.e.,
  \[
     (T_\mu^\nu)^{-1} (x_i)=C_i ,
  \]
  up to a set of measure zero.
\end{theorem}


By the Brenier-McCann theorem~\cite{b,mac} the transport map is unique up to a set of
$\mu$-measure zero. We will deduce that the corresponding power
diagram is also unique from the assumption that the Radon-Nikodym
derivative is strictly positive. The existence of such power diagram
is a result of \cite{AHA} (in a slightly different language), which
relates it to a classical theorem of Minkowski.  Alternatively, it can
be derived from Kantorovich duality.

Denote the power diagram corresponding to $\bar{x}$ and $\bar{r}$ by
$C(\bar{x},\bar{r})$ and the power diagram cell corresponding to $x_i$
by $C_i(\bar{x},\bar{r})$.  We refer to the hyperplane
$H_{ij}(\bar{x},\bar{r}):=\{x\in \Reals^d,f_i(x)=f_j(x) \}$ as the
\emph{bisector} of the points $x_i$ and $x_j$; it is perpendicular to
the line $x_ix_j$.  Finally, we denote the halfspace bounded
by $H_{ij}(\bar{x},\bar{r})$ and containing $C_i$ by
$H^+_{ij}(\bar{x},\bar{r}):=\{x\in \Reals^d,f_i(x)\leq f_j(x) \}$.
By a straightforward calculation, the directed distance from $x_i$ to
the point where the line $x_ix_j$ meets $H_{ij}(\bar{x},\bar{r})$
is given by
\begin{equation}
  \frac{|x_i-x_j|^2-r_j+r_i}{2|x_i-x_j|}. \label{eq1}
\end{equation}

\begin{lemma}
  \label{lem:radii-unique}
 The power diagrams $C(\bar{x},\bar{r})$ and $C(\bar{x},\bar{r}')$ are equal if and only if, for some
  $c\in \Reals$, $r_i'=r_i+c$, for all $i=1,\ldots,n$.
\end{lemma}
\begin{proof}
  It is clear from the definition of the power diagram that shifting
  all radii by a common amount does not affect the diagram.
We now argue that the power diagram determines
  the radii, up to a shift.  Specifically, after arbitrarily setting
  $r_n$ to zero, we show how to reconstruct the remaining radii~$r_i$
  from the power diagram, given site positions $\bar{x}$.
Observe that the graph defined by the adjacencies of cells $C_i$
  across $(d-1)$-dimensional faces is connected. Since $x_i,x_j$ are
  fixed, equation~\ref{eq1} determines the difference
  $r_i-r_j$ for every pair of neighboring cells $C_i$, $C_j$.
  Starting with $r_n=0$, we can uniquely reconstruct all the radii.
\end{proof}

Returning to the transport plan, since $[\frac{d\mu}{d\mathcal{L}^d}]$ is strictly positive, the transformation $T_\mu^\nu$ is unique up to a set of Lebesgue measure zero. The cells of the power diagram are closed sets that coincide with the sets $(T_\mu^\nu)^{-1}(x_i)$ up to a set of measure zero. This uniquely determines the power diagram which by the previous
lemma uniquely determines the radii (up to the normalizing condition
$r_n=0$).

\section{The Idea}
Now we sketch our strategy.  Let $\bar{x}=\langle x_1,x_2, \ldots x_n\rangle$
be an ordered $n$-tuple of pairwise distinct points and $K$ a convex
body. The solution of the optimal transport problem in which the source measure is (an absolutely continuous
approximation of) the volume of~$K$, and the target measure is $\nu_{\bar{x}}=\frac{1}{n}\sum \delta_{x_i}$,
induces a convex partition of $K$ into convex sets of equal
volume. As the points $\langle x_1,x_2, \ldots x_n\rangle$ move, the
surface areas of the corresponding cells change continuously.
We want to argue that for some $n$-tuple they are all equal. To accomplish that, we rely on equivariant topology.
More precisely, consider a map taking an $n$-tuple of distinct points in
$\Reals^d$ to an $n$-tuple of numbers which in coordinate~$i$ takes the value
of the $(d-1)$-dimensional surface area of the intersection of $K$
with the cell $C_i$.  Since the measure $\nu_{\bar{x}}$ is $\Sigma_n$-equivariant, this map is also $\Sigma_n$-equivariant, where $\Sigma_n$ denotes the group of permutations of $n$ elements.
Our result will follow by showing that for any convex body the image
of this map cannot miss the diagonal. We will
derive it from an appropriate Borsuk--Ulam-type statement.

In fact, the previous argument made little reference to surface area.
theorem~\ref{main} will be derived from theorem~\ref{extra}.

Let $\overline{P_2^{{ac}}(\Reals^d)}$ be the closure of $P_2^{{ac}}(\Reals^d)$ under the Wasserstein distance (equivalently, this is the closure under the classical weak topology on measures). 



Since the functionals $G_j$ in theorem~\ref{extra} are assumed to be continuous with respect to
the Hausdorff metric, it suffices to show the theorem for measures in
$P_2^{{ac}^+}(\Reals^d)$ to obtain the result for nice measures. In fact theorem ~\ref{extra} holds for any measure in $\overline{P_2^{{ac}}(\Reals^d)}$ if we substitute the equalities $\mu(K_i)=\frac{\mu(K)}{p^k}$ by inequalities $\mu(K_i)\geq \frac{\mu(K)}{p^k}$ and $\mu(\Int(K_i))\leq \frac{\mu(K)}{p^k}$ where  $K_i$ is topologically closed and $\Int(K_i)$ denotes its interior. When $\mu$ is a nice measure we further know that for every hyperplane $H$, $\mu(H)=0$, by definition of power diagrams we can conclude that $\mu(K_i)=\mu(\Int(K_i))$, and so $\mu(K_i)=\frac{\mu(K)}{p^k}$.

Theorem~\ref{main} follows by setting $G_1$ to be the surface area and $\mu$ to be the measure \[
  \mu(A):=\int_A 1_K d\mathcal{L}^d,
\]
where $A$ is
a measurable set. This measure is nice. It gives zero measure to every hyperplane.  It has a second moment and can be approximated by absolutely continuous measures.
Natural candidates for $G_i$ are given by  the  width, the diameter, or a Minkowski functional. 
Another interesting collection of examples in which theorem~\ref{extra} applies arises in the following fashion. Take a continuous function $g\colon \Reals^d \to \Reals^{l}$, with $d>l$, choose a \emph{centermap} $c$, that is, a continuous functional $c\colon\mathcal{K}^d \to \Reals^d$, and define $G(K):=g(c(K))$. The most obvious choice for $c$ is the barycenter, but  other centers are of interest.
\begin{coro}
Given a nice measure~$\mu$ on $\Reals^d$, a convex body $K\in \mathcal{K}^d$, a prime $p$, an integer $k$, a
function $g\colon\Reals^d \to \Reals^l$, with $d>l$ and a continuous 
centermap $c$. There exist $p^i$ convex sets $K_1,K_2, \ldots K_{p^k}$, such
that
\[
  \mu(K_i)=\frac{\mu(K)}{p^k},
\]
for all $i$, and
\[
  g(c(K_1))=g(c(K_2))=\ldots=g(c(K_{p^k})).
\]
\end{coro}
By taking $g$ to be a 
linear map it is easy to see that theorem~\ref{extra} is quantitatively best possible.
This corollary is similar to the Gromov-Borsuk-Ulam theorem of
\cite{gro}; see also \cite{mem}. In fact,  Gromov's proof was the
starting point for Karasev \cite{k10}. 
Denote by  $\mathcal{K}(\Sphere^d)$ the space of strictly convex sets of the sphere with the Hausdorff metric. A subset of the sphere is \emph{strictly convex} if, for every pair of points in the set, all the minimizing geodesics between them are also contained in the set.
\begin{theorem}[\citenoadjust{\cite{gro}}]
Let $m$ be the Haar measure on the sphere normalized to be a
probability measure. Given a continuous
function $g\colon\Sphere^d \to \Reals^l$, with $d>l$, and a continuous 
centermap $c\colon \mathcal{K}(\Sphere^d)\to \Sphere^d$, there exists a partition of the sphere into $2^k$ convex sets $K_1,K_2, \ldots K_{2^k}\subset \Sphere^{d}$, such
that
\[
  m(K_i)=\frac{1}{2^k},
\]
for all $i$, and
\[
  g(c(K_1))=g(c(K_2))=\ldots=g(c(K_{2^k})).
\]
\end{theorem}
It is not hard to obtain this theorem as a corollary to theorem ~\ref{extra} by embedding the sphere in $\Reals^{d+1}$ and setting the extra degree of freedom of the functional $G$ to approximate a delta mass at the origin.
Alternatively one can work directly with the intrinsic round metric of the sphere $d_{\Sphere^d}$ with the cost function $c(x,y)=-\sin(d_{\Sphere^d}(x,y))$ and a Borsuk-Ulam type statement for configuration spaces on manifolds of ~\cite{k09}.

The following result is a generalization of the Ham Sandwich theorem, which corresponds to the case $n=2$.
\begin{coro}[\citenoadjust{\cite{k10,s}}]\label{sob} Given $d$ nice probability measures $\mu_0, \mu_1,\ldots \mu_{d-1}$ on $\Reals^d$, and \emph{any} number $n$ there is a partition of $\Reals^d$ into convex regions $K_1,K_2, \ldots K_n$ with $\mu_i(K_j)=\frac{1}{n}$ for all $i$ and $j$ simultaneously.\end{coro}
This result was found independently by Karasev \cite{k10} and by Pablo Soberon \cite{s}. It was conjectured by Kaneko and Kano \cite{kk} who also proved the planar version. Soberon proof is an adaptation of our original argument for the case of volume and surface area to the case of several measures; the only algebraic topology tool it relies on is the Borsuk-Ulam theorem for $\ZZ_p$ actions.  
With the full power of theorem~\ref{extra}, Karasev's proof of this corollary is very simple and we now sketch it. 

Write $n=p_1^{\alpha_1} p_2^{\alpha_2}...p_k^{\alpha_k}$ and apply theorem~\ref{extra} with $\mu_0$ as the measure, with $G_i(A):=\mu_i(A)$ and ${p_1}^{\alpha_1}$. For each cell of the partition $K_j$, apply the theorem again to properly renormalized measures, $\mu_i '(A)\colon={p_1}^{\alpha_1} \int_{K_j} A d\mu_i$ to partition each cell into $p_2^{\alpha_2}$ subcells and continue in this manner.

In fact in this proof one can substitute $d-1$ of the measures by $d-1$ normalized continuous valuations with a mild generality condition.

\section{Continuity Lemmas} 

Let the \emph{configuration space} $F_n(\Reals^d)$ be the space $n$-tuples of  of
pairwise distinct labeled points in $\Reals^d$, i.e.,
$F_n(\Reals^d):=\{\langle x_1,x_2, \ldots x_n\rangle: x_i \in
\Reals^d,x_i\neq x_j\}$. Let $\bar{r}=\langle r_1, r_2, \ldots, r_n \rangle$ be an $n$-tuple of numbers.
The power diagram $C(\bar{x},\bar{r})$ has some unbounded cells, on the other hand our statements talk about a partition of a convex body $K$.
By the \emph{truncated} power diagram we mean the intersection of the power
diagram with the convex body $K$. A truncated power diagram can be
thought of as an element of 
\[
  (\mathcal{K}^d)^n:=\mathcal{K}^d \times \mathcal{K}^d \ldots \mathcal{K}^d
\]
with the $l_\infty$  metric, i.e., 
\[
  d((K_1,K_2, \ldots K_n), (K'_1,K'_2, \ldots K'_n))=\max_i
  d_H(K_i,K'_i),
\]
with $d_H$ the Hausdorff distance. 
Lemma~\ref{lemma1} below shows the continuity of the map 
\[
  C\colon F_n(\Reals^d)\times \Reals^{n-1} \to (\mathcal{K}^d)^n
\]
given by the truncated power diagram $C(\bar{x},\bar{r})$.

By  theorem~\ref{AHA},
for any $q\in \Reals^n$, $q_i \geq 0$, any
point $\bar{x} \in F_n(\Reals^d)$, and any absolutely continuous
probability measure $\mu$, provided $\sum_i q_i=\mu(K)$,  there exists a unique (up to a shift) $n$-tuple of radii
$\bar{r}=\langle r_1, r_2, \ldots, r_n \rangle$ such that cells of the
power diagram with sites at $\bar{x}=\langle x_1,x_2, \ldots
x_n\rangle$ and radii $\bar{r}$, truncated to within $K$, have
$\mu$-measures $\langle q_1,q_2, \ldots q_n \rangle$. Fixing $q$ and $\mu$, theorem~\ref{AHA} describes a map sending a point in $F_n(\Reals^d)$ to
a power diagram. We will show that this map is continuous. By lemma~\ref{lemma1} it is enough to show that the assignment of an $n$-tuple of radii $\bar{r}$ to each $\bar{x}$ 
is continuous (lemma~\ref{lemma2}). 
Incidentally, observe that the condition $\sum_i q_i=\mu(K)$ and the fact that we can assume
$r_n=0$ imply that both $\bar{r}$ and $q$ are $(n-1)$-dimensional.

Now we work out the details of lemma~\ref{lemma1}.
We will use the term
\emph{vanishing point} to refer to the points $(\bar{x},\bar{r})$ in
$F_n(\Reals^d)\times \Reals^{n-1}$ where the corresponding power diagram
has a \emph{vanishing cell}, that is, every open neighborhood of
$(\bar{x},\bar{r})$ contains points where the cell is nonempty, but at
$(\bar{x},\bar{r})$ it is empty.  Note that vanishing points form a
closed subset of $F_n(\Reals^d)\times \Reals^{n-1}$.  We will need the
following (easy) fact which is a special case of theorem~1.8.8 in
\cite{convexityBook}.
\begin{fact}
  \label{fact:convergence}
  If $A$ and $B$ are convex bodies, such that $A \cap B$ has non-empty
  interior, and the sequences $A_k$ and $B_k$ of convex sets converge, respectively,
  to $A$ and $B$, with respect to the Hausdorff metric, then $A_k \cap
  B_k$ converges to $A \cap B$.
\end{fact}


\begin{lemma}\label{lemma1} 
  For any convex body $K$, the map $C$ that assigns to each point
  $(\bar{x},\bar{r}) \in F_n(\Reals^d)\times \Reals^{n-1}$ the convex
  partition corresponding to its truncated power diagram $C(\bar{x},\bar{r})$ is continuous
  whenever $(\bar{x},\bar{r})$ is not a vanishing point.  
\end{lemma}

\begin{proof}
  Slightly abusing the notation, put $C_i:= \bigcap_{j\in V-\{i\}}
  (H_{ij}^+ \cap K)$ (we will repeat this abuse without warning).  Each $H_{ij}^+ \cap K$, when non-empty, depends
  continuously on $(\bar{x},\bar{r})$, with the dependence given
  by~\eqref{eq1}.  Since $(\bar{x},\bar{r})$ is a non-vanishing point,
  $C_i$ is either identically empty in a neighborhood of
  $(\bar{x},\bar{r})$ or has non-empty interior.  Therefore, by
  repeated application of Fact~\ref{fact:convergence}, $C_i$ varies
  continuously with $(\bar{x},\bar{r})$.
 \end{proof}

Now we turn to lemma~\ref{lemma2} which follows from
the continuity of the dependence of the transport maps on the target
measure (known as stability of the transport map).

The space of $L_2$-transformations
\[
L_2(\Reals^d,\Reals^d,\mu)= \bigl\{T:\Reals^d \to
\Reals^d : \int |x-T(x)|^2 d\mu(x)<\infty\bigr\}
\]
 will be metrized by $d_{L_2(\mu)}(T_1,T_2)=(\int |T_1(x)-T_2(x)|^2 d\mu(x))^{1/2}$ and
the space of Borel probability measures with a second moment $\mathcal{P}_2(\Reals^d)$
by the Wasserstein distance.
 It is well known that for a fixed absolutely continuous
 source measure $\mu$, the map
  $\mathcal{T}\colon\mathcal{P}_2(\Reals^d) \to
  L_2(\Reals^d,\Reals^d,\mu)$ given by $\nu \to T_\mu^\nu$ is
  continuous, see \cite{v08}~and~\cite{ags}.

From this point on, we
assume that $q_i>0$, for all $i$. 

\begin{lemma}
  \label{lemma2}
 Let $r_{K,q} \colon F_n(\Reals^d) \to \Reals^{n-1}$
be the tuple of radii assigned to the points $\langle x_1, x_2 \ldots
x_n\rangle$ as in theorem~\ref{AHA}, so that
$\mu(C_i(\bar{x},\bar{r})\cap K)=q_i\mu(K)$. The map $r_{K,q}$ is continuous.
\end{lemma}
\begin{proof}
 Without loss of generality we assume $\mu(K)=1$ and $supp(\mu) \subset K$. Recall that $r:=r_{K,q}$ maps $F_n(\Reals^d)$ to $\Reals^{n-1}$, and in
  both spaces we are using topologies homeomorphic to the one
  inherited from Euclidean distance. 
  By the characterization of theorem~\ref{AHA},
  lemma~\ref{lem:radii-unique}, and the surrounding discussion,
  $T_\mu^\nu$ is given by a power diagram with sites at
  $\mathit{supp}(\nu)$ and a unique tuple of radii with the
  normalizing condition $r_n=0$.

  Take a sequence $\bar{x}^k$ converging to $\bar{x}$. The measures
  $\nu^k=\sum_{x_i\in\bar{x}^k} q_i \delta_{x_i}$ converge to
  $\nu=\sum_{x_i\in\bar{x}} q_i \delta_{x_i}$ and therefore
  $T_\mu^{\nu^k}$ converges to $T_\mu^\nu$.  Each map uniquely defines
  a set of radii, so we have a sequence $r(\bar{x}^k)$.  Assume, for a
  contradiction, that it does not converge to $r(\bar{x})$. 
  Since $\bar{x}^k$ converges, for all $\delta>0$, 
  $|x_i^k -x_i|<\delta$, for all $i$ and large enough $k$.  If a bisector
  misses $K$ entirely, one of the truncated cells must be empty,
  contradicting our assumption that $q_i>0$.   Therefore, the numbers
  $|r_i^k|$ are bounded, by eq.~\eqref{eq1}.  
  
  Thus, a subsequence of $r(\bar{x}^k)$ converges to some $r'\neq
  r(\bar{x})$.  Abusing notation, denote by $T_\mu^{r'}$ the
  transformation associated with the power diagram
  $C(\bar{x},\bar{r}')$, which must differ from $C(\bar{x},\bar{r})$
  by lemma~\ref{lem:radii-unique}.  Then
  \begin{align*}
    \int |T_\mu^{r'} (x)-T_\mu^\nu(x)|^2 d\mu (x) & =\sum_{j,i}
    \mu(C_i(\bar{x},r')\cap
    C_j(\bar{x},r))|x_i-x_j|^2 \\ 
    & \geq |x_{i_0}-x_{j_0}|^2\mu(C_{i_0}(\bar{x},r')\cap
    C_{j_0}(\bar{x},r))>0,
  \end{align*}
  for some pair $(i_0,j_0)$, which contradicts the continuity of
  transport maps.
\end{proof}

\section{Equivariant Topology}

We are interested in the case $q=\frac {\mu(K)}{n}
(1,1,\ldots,1)$, i.e. $\nu = \frac{\mu(K)}{n} \sum_{x_i \in \bar{x}} \delta_{x_i}$, where $\delta_{x_i}$ is the delta mass at $x_i$. Consider  the map 
\begin{align*}
  P(\bar{x}):= 
  (&G_1(C_1(\bar{x},r(\bar{x}))), G_1(C_2(\bar{x},r(\bar{x}))), \ldots, G_1(C_n(\bar{x},r(\bar{x}))),\\
  &G_2(C_1(\bar{x},r(\bar{x}))), G_2(C_2(\bar{x},r(\bar{x}))), \ldots, G_2(C_n(\bar{x},r(\bar{x}))),\\
  &\qquad\cdots\\
  &G_{d-1}(C_1(\bar{x},r(\bar{x}))), G_{d-1}(C_2(\bar{x},r(\bar{x}))), \ldots, G_{d-1}(C_n(\bar{x},r(\bar{x})))),
\end{align*}
where $r(\bar{x})$ corresponds to the optimal transport as before.
We want to show that the map $P$
meets the $(d-1)$-dimensional diagonal,  
\begin{align*}
\Delta_{d-1}:=\{(&t_1,t_2, \ldots, t_{d-1},\\
&t_1 ,  t_2 ,\ldots, t_{d-1} ,\\
&\qquad\cdots \\
&t_1,t_2, \ldots, t_{d-1}): t_i\in \Reals\}.\end{align*} 
Since we chose $q=\frac {\mu(K)}{n}
(1,1,\ldots,1)$
this map is $\Sigma_n$-equivariant, where the action on $F_n(\Reals^d)$ relabels the points and the action on $\Reals^{n(d-1)}$
permutes blocks of $d-1$ coordinates.  Assuming that $P$ avoids 
$\Delta_{d-1}$ implies that we can equivariantly deformation retract it to a
$\Sigma_n$-equivariant map from $F_n(\Reals^d)$ to $\Reals^{(n-1)(d-1)}$, and the resulting map is nowhere zero. This strong
deformation retract is just an affine interpolation
between the identity and the orthogonal projection on the orthogonal complement of $\Delta_{d-1}$. Since we assumed that the map avoided $\Delta_{d-1}$, the image of our deformation retract is contained in $\Reals^{(d-1)(n-1)} \setminus \{0\}$. We can now deformation retract it to a sphere $\Sphere^{(n-1)(d-1)-1}$ by linearly interpolating with the map $x \to \frac{x}{|x|}$. We would reach a contradiction if we can prove that,
 \emph{there is no  $\Sigma_n$-equivariant map} $$\phi:F_n(\Reals^d)\to \Sphere^{(d-1)(n-1)-1}.$$
An easy way to achieve this result is to let $n$ be a prime smaller
than the dimension~$d$, restrict the action to a $\ZZ_n$ action and
apply Dold's theorem (theorem~\ref{d} on the appendix); this was our
approach before we were aware of the following theorem which is
contained in \cite{k09}. In fact, this result is basically contained
in the paper \cite{v} of Vasil'ev and the main idea of the proof goes
back to Fuks \cite{f} who solved the simplest case of $n=2^k$ and $d=2$.

\begin{theorem}\label{top} Let $p$ be a prime, $k\geq 1$ an integer, and $n=p^k$. For any $\Sigma_n$-equivariant map $f\colon F_n(\Reals^d) \to \Reals^{(d-1)(n-1)}$, there exists a configuration $\bar{x}\in  F_n(\Reals^d)$, such that $f(\bar{x})=\bar{0}$. \end{theorem}

Before going to the proof of this theorem we recall some generalities about Characteristic Classes, see \cite{mil}, \cite{bott}, and \cite{bre} for more details.

Given a vector bundle $\pi \colon E \to B$, with fiber $\Reals^m$, let $i:B \to E$ be the zero section $i(b):=(b,0_b)$ and  $E_0:=E \setminus i(B)$. The Thom class of the bundle $E\xrightarrow{\pi} B$ is a uniquely defined cohomology class $[u] \in H^m(E, E_0)$ that restricts to an orientation class on each fiber. Thom showed that for any vector bundle the Poincare Dual of the zero section with respect to the pair $(E, E_0)$ is the Thom class and moreover for any cohomology group, the map taking $[b] \in H^k(B)$ to $\pi^*[b] \cup [u] \in H^{k+m}(E,E_0)$ is an isomorphism (where $\pi^*[b]$ is thought as a cohomology class of the pair $(E,E_0)$ once we observe that we can pick a representative that vanishes outside a neighborhood of the zero section). A very useful invariant of a vector bundle is its \emph{Euler class} $\chi(E)$, a cohomology class in $H^m(B)$ defined as the image of the Thom class  induced by the inclusion $i\colon (B, \emptyset) \to (E, E_0)$, i.e. $\chi(E):=i^*([u]) \in H^m(B)$. Equivalently, $\chi(E)$ is the class that goes to $[u] \cup [u]\in H^{2m}(E,E_0)$ under the Thom isomorphism. We will use two properties of the Euler class:
\begin{enumerate}

\item[(P1)] If there is a nonvanishing section $s\colon B \to E$, then the Euler class $\chi(E)$ is trivial. 

\item[(P2)] If $s:B \to E$ is a transversal section and the base space $B$ is a manifold, then the homology class of the zero set $s^{-1}(0)$ is Poincare Dual to the Euler class. 
\end{enumerate}

In our case the base space is not a manifold but a version of property ~(P2) still holds; we clarify this issue below.
As before, let $\rho' \colon \Sigma_n \to O(n)$ be the standard representation of $\Sigma_n$ by permutation matrices. The diagonal $\Delta:=(t,t, \ldots t)$ is invariant under the induced action of $\Sigma_n$ on $\Reals^n$ and this representation splits into two irreducible representations, the diagonal~$\Delta$, and its orthogonal complement, denoted $\Delta^*:=\{(y_1, y_2, \ldots y_n): y_i \in \Reals, \sum y_i=0\}$.  We look at the irreducible representation on the orthogonal complement and denote it by $\rho \colon \Sigma_n \to O(n-1)$.
Associated to this representation there is an $(n-1)$-dimensional vector bundle $\eta$ given by $(\Delta^* \times F_n(\Reals^d))/ \Sigma_n  \to F_n(\Reals^d)/ \Sigma_n$, where the action on  $\Delta^*$ is the one induced by $\rho$. Any  $\Sigma_n$-equivariant map  $f:F_n(\Reals^d)\to \Delta^*$ corresponds to a section of $\eta$. Similarly, a $\Sigma_n$-equivariant map $f:F_n(\Reals^d)\to (\Delta^*)^{(d-1)}$ corresponds to a section of the $(d-1)$-fold Whitney sum of $\eta$ with itself. Theorem~\ref{top} is equivalent to the nonexistence of nowhere-zero sections of the vector bundle $\eta^{d-1}$. This will be derived from property~(P1) above, by showing that the Euler class $\chi(\eta^{d-1})$ (under the right choice of coefficients) is not zero, which in turn, will be derived from property~(P2).

\begin{proof}
We will describe a cell decomposition of the one-point
compactification of $F_n(\Reals^d)$, denoted by
$F_n'(\Reals^d):=F_n(\Reals^d) \cup \{pt\}$. This decomposition is labeled by elements of $\Sigma_n$ and the labeling is 
$\Sigma_n$-equivariant so it induces a cell decomposition of
$(F_n'(\Reals^d)/ \Sigma_n,pt)$ when we ``forget'' the labels. We will use these decompositions to perform
(co)homological calculations. These decompositions appear in Fuks paper \cite{f} for $d=2$ and in \cite{k09} in general. 
The reader is encouraged to examine figure~\ref{fig:tree} and skip the
next few paragraphs.
\begin{figure}
  \centering
  \includegraphics[scale=0.4]{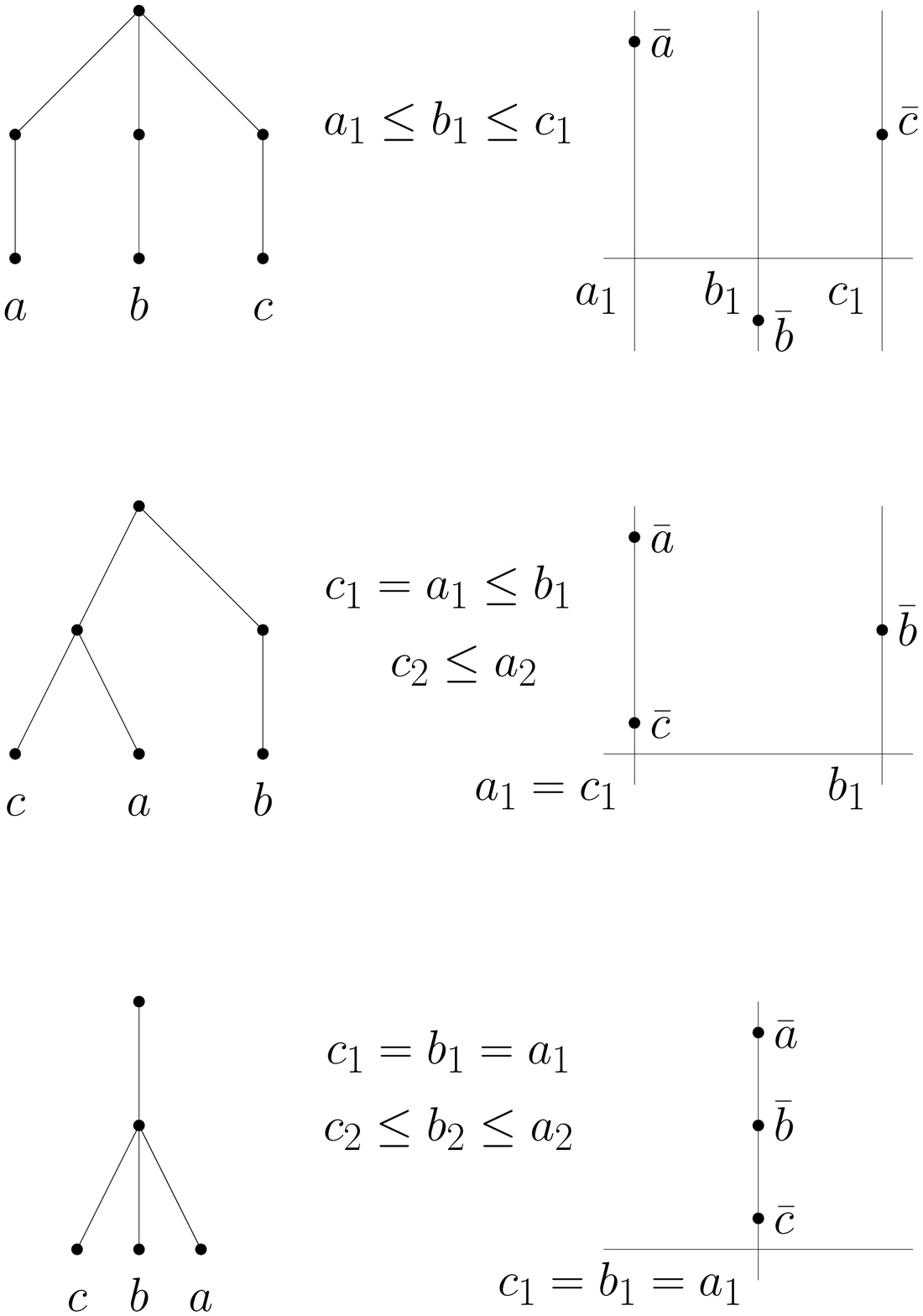}

  \caption{Trees and cells: three trees and
    matching example point configurations from the corresponding cells}
  \label{fig:tree}
\end{figure}

The cell decomposition of $F_n'(\Reals^d)$ has one $0$-cell that corresponds to the point at infinity. The remaining cells are in bijection with a family of ordered, labeled trees of the following form:
the height of the tree is $d$, in other words,  the tree
has $d+1$ levels including the root. Children of every node form a linearly ordered set. The tree has $n$ leaves, all of which
occur at the bottom level. Only the leaves are labeled and they are labeled with numbers $1$ through $n$ (we can think of these labels as an element of $\Sigma_n$).

The bijection between trees and cells of the decomposition is such that the dimension of the cell corresponding to the tree $T$ is $|T|-1$, where $|T|$ denotes the number of vertices in the tree $T$. The attaching maps will be defined implicitly. Instead we describe maps from the cells to configuration space. The following sets will be convenient to describe the maps, 
$$\Box^k:=\{(t_1, t_2, \ldots t_k) \in [-\infty,\infty]^k :
t_1\leq t_2\ldots \leq t_k\}.$$ Note that $\Box^k$ is homeomorphic to a closed $k$-dimensional ball. Denote by $deg(v)$ the number of children of the vertex $v$, for example a leaf has $deg(v)=0$.
For each tree $T$ we will define a continuous map $C_T \colon \Pi_{v \in V} \Box^{deg(v)} \to F_n'(\Reals^d)$, where $\Pi_{v \in V} \Box^{deg(v)}$ denotes the Cartesian product of $\Box^k$. To describe this map we recall a sorting algorithm that assigns a tree to each configuration. 

The root
(level~0) is associated with the entire configuration which is an $n$-tuple of points in $\Reals^d$. We
sort the points by their first coordinate; nodes on level~$1$
correspond to groups of points sharing the same coordinate.  For
example, if all points in the configuration have the same first coordinate, the root
has one child; if all first coordinates are different, the root has $n$
children, sorted in the order of coordinate values.  The construction
proceeds recursively: on level~$2$, we consider the set of points
associated with a node at level~$1$, split them into groups according
to the value of their second coordinate, sorted in increasing order,
and associate a level~$2$ node with each group.  Repeat the process
for each level, stopping at level~$d$, where we necessarily get a
total of $n$ leaves, as all points differ in at least one coordinate. Finally, label the leaves of the tree by the labels of the points of the configuration. This labeling of the leaves corresponds to the lexicographical order of the points of the configuration.
In a word, given two points $x_{i_1}$ and $x_{i_2}$ in a configuration, the closest common ancestor of the leaves labeled $i_1$ and $i_2$ represents the largest coordinate in which the two points coincide. 
With this sorting algorithm in mind, we return to describing the map $C_T$. In some sense, the inverse of the algorithm. 

There is a natural correspondence between the coordinates of $\Pi_{v \in V(T)} \Box^{deg(v)}$ and the vertices of $T$ minus the root. For each element $q \in \Pi_{v \in V(T)} \Box^{deg(v)}$  we think of an assignment of a real number to each vertex of the tree. For example, the first $deg(root)$ coordinates of $q$ are assigned to the children of the root respecting the order. Given this assignment we describe the configuration $C_T(q)$. The $j$-th coordinate of the $i$-th point of the configuration is the coordinate of $q$ assigned to the unique vertex of level $j$ on the path from the root to the leaf ~i. This process describes an element of $\Reals^{nd}$ corresponding to an element of $\Pi_{v \in V(T)} \Box^{deg(v)}$, for $F_n'(\Reals^d)$ to be the image, redefine this assignment to map all the elements that got mapped to $\Reals^{nd}\setminus F_n(\Reals^d)$ to the point at infinity . Is easy to check that we have defined a cell decomposition. Note that the boundary of the $(n+d-1)$- dimensional cell is the point at infinity. The elements at the boundary of any cell are those for which the inequality between two coordinates $t_{j_1}\leq t_{j_2}$ in one of the sets $\Box^{deg(v)}$ becomes equality or those for which some coordinate $t_j$ is $\pm \infty$. The construction guarantees that boundary points of cells of dimension larger than $n+d-1$ are mapped to the lower dimensional skeleton. This concludes the description of the cell decomposition.

This decomposition is $\Sigma_n$-equivariant, moreover, the only fixed point is the point at infinity. It induces a cell decomposition on the pair $(F_n'(\Reals^d)/\Sigma_n,pt)$, where $pt$ is the point at infinity. One can alternatively work with the pair $(F_n^\epsilon(\Reals^d)/\Sigma_n, \partial F_n^\epsilon(\Reals^d)/\Sigma_n)$ where $F_n^\epsilon(\Reals^d)/\Sigma_n$ is the space of unlabeled configurations such every two points are at distance at least $\epsilon$ from each other and each of them is at distance more than $\epsilon$ from the point at infinity. 

We will denote by $\bold{C_i}(F_n'(\Reals^d)/\Sigma_n,pt;R)$ the chain
complex corresponding to the cell decomposition of
$(F_n'(\Reals^d)/\Sigma_n,pt)$ with coefficients in $R$ (if we omit
specifying the coefficients we mean $R=\ZZ$). Forgetting the labels,
there is only one tree with $n+d$ vertices so $\bold{C_{n+d-1}}(F_n'(\Reals^d)/\Sigma_n,pt)$ has only one cell generator which we will denote by $e$. There are no trees with fewer vertices, so the kernel of the boundary map $\partial_{n+d-1}$ is $\bold{C_{n+d-1}}(F_n'(\Reals^d)/\Sigma_n, pt)$; the boundary of the cell $e$ is attached to the point at infinity. To compute the $(n+d-1)$-homology group we just need to understand what is the image of the boundary operator. The cells generating $\bold{C_{n+d}}(F_n'(\Reals^d)/\Sigma_n, pt)$ correspond to trees with one vertex of each rank from $0$ to $d-2$, two vertices of rank $d-1$ and $n$ vertices of rank $n$. Correspondingly these are configurations lying on the same $2$-plane with the points of the configuration divided into two groups, each group lying on a line, so $\bold{C_{n+d}}(F_n'(\Reals^d)/\Sigma_n,pt)$ has $n-1$ generators, one for each nontrivial integer solution of $n_1+n_2=n$. The points lie on lines $l_1$ and $l_2$ parallel to the last coordinate axis, we assume that $l_1$ is lexicographically before $l_2$, and we let $n_1$ be the number of points on $l_1$ and $n_2$ be the number of points on $l_2$. In this way, specifying the value of $n_1$ determines a generator of $\bold{C_{n+d}}(F_n'(\Reals^d)/\Sigma_n,pt)$, and the set of these $n-1$ generators forms a basis.  Since $\mathop{\mathrm{ker}} \partial_{n+d-1}=\bold{C_{n+d-1}}(F_n'(\Reals^d)/\Sigma_n,pt)$, understanding the $(n+d-1)$-homology group boils down to understanding $\mathop{\mathrm{Im}} \partial_{n+d}$, here the choice of coefficients will be crucial.

Now we exhibit a section $s_g$ of this bundle
that is transversal to the zero section and such
that the pull back of the zero section by $s_g$ is the cell
$e$. Consider the map $g \colon F_n(\Reals^d)\to \Reals^{n(d-1)}$ that
forgets the last coordinate of every point of the configuration.  This
map is clearly equivariant and so, as before, it induces a
section $$s_g \colon F_n(\Reals^d)/\Sigma_n \to (F_n(\Reals^d) \times
\Reals^{(d-1)(n-1)})/ \Sigma_n.$$ The manifold $s_g^{-1}(0)$
corresponds to the configurations for which all the points share the
first $d-1$ coordinates---this is precisely the generating cell in
$\bold{C_{n+d-1}}(F_n'(\Reals^d)/ \Sigma_n,pt)$. To use fact~(P2) above
and conclude that the Euler class is not trivial, observe that the
section is transversal to the zero section. Since the covering
$F_n(\Reals^d) \to F_n(\Reals^d)/ \Sigma_n$ is regular we can
equivalently show that the image of $(id,g) \colon F_n(\Reals^d) \to
F_n(\Reals^d) \times \Reals^{n(d-1)}$ is transversal to the image of
$(id, 0)$, moreover we can identify $F_n(\Reals^d)$ with its inclusion
in $\Reals^{nd}$. Now we have two linear maps and transversality
follows from counting dimensions. We can conclude that $s_g$ is transversal to the zero section. 
The rest of the argument consists of ensuring conditions so that we can invoke property ~(P2). Provided that the pair $(F_n'(\Reals^d), pt)$ represents the generator of the fundamental class in the homology $H_{nd}(F_n(\Reals^d), pt;R)$, the map $$H^{(n-1)(d-1)}(F_n'(\Reals^d)/\Sigma_n)\to H_{n+d-1}(F_n'(\Reals^d)/\Sigma_n,pt) $$ given by $[\alpha]\to [(F_n'(\Reals^d)/\Sigma_n,pt)]\cap [\alpha]$, sends the Euler class to the class~$[e]$ in homology. Here $\cap$ denotes the cap product and $[(F_n'(\Reals^d)/\Sigma_n,pt)]$ denotes the generator of the fundamental class of the pair. 

We need to show that, for each prime $p$, there is a choice of coefficients for which the class $[e]$ is not trivial and configuration space is orientable. Below, to keep track of orientations, we will give a sketch of the proof  of the version of property ~(P2) that we are using.
 
For $p=2$, we take coefficients in the field $\ZZ_2$. The boundary map
is given by $$\partial(e_{n_1})={n \choose n_1}.$$ This is easy to
see:  consider a configuration $\bar{x}$ in the cell $e$.  This is a
configuration of unlabeled points on a line $l$.  Every configuration is a regular value of the attaching map. Fix a cell $e_{n_1}\in C_{n+d}(F_n'(\Reals^d)/\Sigma_n, pt; \ZZ_2)$ and note that there are exactly ${n \choose n_1}$ configurations at the boundary of $e_{n_1}$ that map to $\bar{x}$, one for each splitting of $\bar{x}$ into two sets of $n_1$ on the left line and ${n-{n_1}}$ on the right line. Now, if $n=2^k$, then  $2\mid{n \choose n_1}$ and so the boundary map $\partial$ is the zero map in $\ZZ_2$ coefficients and so $H_{n+d-1}(F_n'(\Reals^d)/\Sigma_n,pt; \ZZ_2)=\ZZ_2$.
Moreover, since $F_n(\Reals^d)/\Sigma_n$ is a manifold, looking at the long exact sequence of the pair $(F_n'(\Reals^d)/\Sigma_n,pt)$ we obtain
$H^{nd}(F_n'(\Reals^d)/\Sigma_n,pt;\ZZ_2)=\ZZ_2$, and we can
conclude that theorem~\ref{top} holds for $n=2^k$; see \cite{f}.

There is an alternative approach to the case $p=2$ that can be found
in \cite{gro} and \cite{mem}. Instead of looking at the full group of
symmetries, restrict the action to the automorphism group of a
$3$-regular tree on $k$ levels.  This group sits naturally inside the
symmetric group. It is not hard to prove the inductive formula
$Aut(T_k)=\ZZ_2 \wr Aut(T_{k-1})$, where $\wr$ denotes the wreath
product. The Euler class with $\ZZ_2$ coefficients coincides with the
top Stiefel-Whitney class which is amenable to induction on the depth
of the tree. The base of the induction corresponds to the Borsuk-Ulam theorem (see of page 10 of \cite{mem}). 
For primes other than $2$ we need to use homology with twisted coefficients. The number of points is odd. For $d$ even, the manifold $F_n(\Reals^d)/\Sigma_n$ is orientable and the pair $(F_n'(\Reals^d), pt)$ represents a generator in the top homology. For $d$ odd, $F_n(\Reals^d)/\Sigma_n$ is nonorientable. We now recall the definition of homology with twisted coefficients and prove that when $F_n(\Reals^d)/\Sigma_n$ is nonorientable, the pair $(F_n'(\Reals^d), pt)$ represents a non-trivial homology class in the top dimension. This part of the proof works for any nonorientable manifold. If $F_n(\Reals)/\Sigma_n$ is not $\ZZ$-orientable, there is a connected two-sheeted covering space $F_n(\Reals^d)/A_n \to F_n(\Reals^d)/\Sigma_n$ with $\ZZ_2$ acting on it by deck transformations. This action of $\ZZ_2$ corresponds to the sign map $\Sigma_n \to \{+1,-1\}$, where $\{+1,-1\}$ has a multiplicative structure isomorphic to $\ZZ_2$. We denote the associated chain complex by $\bold{C_i}(F_n'(\Reals^d)/ A_n, pt)$. 
The nontrivial element $\tau \in \ZZ_2$ acts as an involution on $F_n(\Reals^d)/ A_n$, so it induces a natural $\ZZ[\ZZ_2]$ module structure on $\bold{C_i}(F_n'(\Reals^d)/ A_n, pt)$. Recall that the group ring $\ZZ[\ZZ_2]$ consists of the formal sums $\{a+b\tau \mid a,b \in \ZZ\}$, where the addition is given by $(a+b\tau)+(a'+b'\tau)=(a+a')+(b+b')\tau$ and the multiplication by $(a+b\tau)(a'+b'\tau)=(aa'+bb')+(ab'+a'b)\tau$. Note that the boundary maps $\partial$ become $\ZZ[\ZZ_2]$-module homomorphisms. 
The group $\ZZ_2$ acts on $\ZZ$, by $\tau(m)=-m$ for every $m \in \ZZ$. And so, it can also be given a $\ZZ[\ZZ_2]$-module structure by the rule $(a + b\tau)m=ma-mb$. 
Consider the $\ZZ[\ZZ_2]$-modules $\bold{C_i}(F_n'(\Reals^d)/ A_n,pt) \otimes_{\ZZ[\ZZ_2]} \ZZ$, where the action on $\bold{C_i}(F_n'(\Reals^d)/ A_n,pt)$ is induced by the deck transformation. We denote this module by $\bold{C_i}(F_n'(\Reals^d/\Sigma_n,pt);\widehat{\ZZ})$ and consider the operator $$\widehat{\partial}:=\partial \otimes 1 \colon \bold{C_i}(F_n'(\Reals^d)/\Sigma_n,pt;\widehat{\ZZ}) \to \bold{C_{i-1}}(F_n'(\Reals^d)/\Sigma_n,pt;\widehat{\ZZ}).$$ This is a boundary operator in the sense that $\widehat{\partial}^2=0$.  The corresponding homology groups are the homology groups with twisted coefficients $\widehat{\ZZ}$; we will denote them by $H_i(F_n'(\Reals^d)/\Sigma_n,pt;\widehat{\ZZ})$. 
Let us start showing that $(F_n'(\Reals^d)/\Sigma_n,pt)$ represents a generator for twisted coefficients whenever $F_n(\Reals^d)/\Sigma_n$ is not $\ZZ$-orientable, that is, $H_{nd}(F_n'(\Reals^d)/\Sigma_n, pt; \widehat{\ZZ})=\ZZ$ . Consider the following decomposition of $\bold{C_i}(F_n'(\Reals^d)/ A_n, pt;\widehat{\ZZ})$ into submodules: Let $\bold{C_i}^+(F_n'(\Reals^d)/ A_n, pt)$ be the submodule that consists of chains of the form $\sigma+ \tau(\sigma)$ and similarly $\bold{C_i}^-(F_n'(\Reals^d)/ A_n, pt)$ be the chains of the form $\sigma- \tau(\sigma)$ where $\sigma$ runs through all chains in $\bold{C_i}(F_n'(\Reals^d)/ A_n, pt;\ZZ)$. Note that this is just a decomposition into symmetric and antisymmetric parts with respect to $\tau$, i.e. $\tau(\sigma+ \tau(\sigma))=\sigma + \tau(\sigma)$ and  $\tau(\sigma- \tau(\sigma))=-\sigma+ \tau(\sigma)$. It is clear that $$\bold{C_i}(F_n'(\Reals^d)/ A_n,pt)=\bold{C_i}^+(F_n'(\Reals^d)/ A_n,pt)\oplus \bold{C_i}^-(F_n'(\Reals^d)/ A_n,pt).$$
Note that the following short sequences are exact: $$0 \to \bold{C_i}^+(F_n'(\Reals^d)/ A_n,pt) \to \bold{C_i}(F_n'(\Reals^d)/ A_n,pt) \xrightarrow{\phi_-} \bold{C_i}^-(F_n'(\Reals^d)/ A_n,pt) \to 0,$$ where $\phi_-(c):=c-\tau(c)$; and  $$0 \to \bold{C_i}^+(F_n'(\Reals^d)/ A_n,pt) \to \bold{C_i}(F_n'(\Reals^d)/ A_n,pt) \xrightarrow{\pi} \bold{C_i}(F_n'(\Reals^d)/ A_n,pt) \otimes_{\ZZ[\ZZ_2]} \ZZ \to 0,$$ where $\pi$ is the quotient map.
From this we can conclude that $\bold{C_i}^-(F_n'(\Reals^d)/ A_n,pt)$ and $\bold{C_i}(F_n'(\Reals^d)/ A_n,pt) \otimes_{\ZZ[\ZZ_2]} \ZZ$, are isomorphic as modules.
Putting this together with the short exact sequence $$0 \to \bold{C_i}^-(F_n'(\Reals^d)/ A_n,pt) \to \bold{C_i}(F_n'(\Reals^d)/ A_n,pt) \xrightarrow{\phi_+} \bold{C_i}^+(F_n(\Reals^d)/ A_n,pt) \to 0,$$ where $\phi_+(c):=c+\tau(c)$,  we obtain a long exact sequence $$\ldots \to H_i(F_n'(\Reals^d)/\Sigma_n,pt;\widehat{\ZZ})) \to H_i(F_n'(\Reals^d)/ A_n,pt;\ZZ) \to H_i(F_n'(\Reals^d)/\Sigma_n,pt;\ZZ) \to \ldots. $$ Substituting $i=nd$ we obtain an isomorphism $$H_{nd}(F_n'(\Reals^d)/ \Sigma_n,pt;\widehat{\ZZ}) \to H_{nd}(F_n'(\Reals^d)/ A_n,pt;\ZZ),$$ whenever $F_n(\Reals^d)/ \Sigma_n$ is not orientable, but $F_n(\Reals^d)/ A_n$ is orientable, so $H_{nd}(F_n'(\Reals^d)/\Sigma_n,pt;\widehat{\ZZ})=\ZZ$. 

In this generality, cup product with the fundamental class of the pair $(F_n'(\Reals^d)/\Sigma_n,pt)$ defines a homomorphism
$$H^{nd-k}(F_n(\Reals^d)/\Sigma_n;R) \to H_k(F_n'(\Reals^d)/\Sigma_n,pt;R \otimes or(F_n(\Reals^d)/\Sigma_n))$$  
where $R$ is a system of coefficients and $or(F_n(\Reals^d)/\Sigma_n)=\widehat{\ZZ}$ if $n$ and $d$ are odd and $or(F_n(\Reals^d)/\Sigma_n)=\ZZ$ if $n$ is odd and $d$ is even.  

Now we will show that for twisted coefficients 
\begin{equation}
  \label{eq:coeff}
  \partial(e_{n_1})=(-1)^{n_1} {n \choose n_1},
\end{equation}
where $e_{n_1}$ is the cell of configurations with $n_1$ points on $l_1$ and $n-n_1$ points on $l_2$, two standard lines on a standard plane. Assuming the formula, we see that similarly to the case $p=2$, if $n=p^k$ then, $p \mid  {n \choose n_1}$ (whenever $n_1 \neq 0, n$) and so $\mathop{\mathrm{Im}}\partial_{n+d}=p \ZZ$ and $[e] \in H_{n+d-1}(F_n'(\Reals^d),pt;\widehat{\ZZ})$ is not trivial. On the other hand, if $n$ is not of the form $p^k$, then this map is surjective and $H_{n+d-1}(F_n'(\Reals^d),pt;\widehat{\ZZ})$ is trivial. 
We now prove the claimed formula for the boundary map. The only difference with the $\ZZ_2$ case is that we have to be careful with possible cancellations of the attaching map. The cell $e \in \bold{C_{n+d-1}}(F_n'(\Reals^d)/\Sigma_n,pt)$ lifts to two cells $e'$ and $e''$ in $\bold{C_{n+d-1}}(F_n'(\Reals^d)/A_n,pt)$. When we tensor with $\ZZ$ via $\ZZ_2$, the projections of these cells become linearly dependent, namely $(e'+e'')\otimes 1=e' \otimes 1+e'' \otimes 1=e' \otimes 1+ e' \otimes (-1)=e'\otimes (1-1)=0$ and similarly for cells $e_{n_1}'$ and $e_{n_1}''$ in $\bold{C_{n+d}}(F_n(\Reals^d)/A_n,pt)$.
To analyze the boundary map in the case of twisted coefficients we look at differential forms on labeled configuration space $\bold{C_{n+d}}(F_n'(\Reals^d),pt)\to \bold{C_{n+d-1}}(F_n'(\Reals^d),pt)$. As before the fiber of every labeled configuration under the attaching map consists of ${n \choose n_1}$ configurations, all of them coming from distinct cells. 
Two labeled configurations in the same fiber are related by an element of $\Sigma_n$. 
Let $\sigma:[n] \to[n]$ be the permutation defined by the lexicographical order of the configuration, so that  $x_{\sigma(1)}< x_{\sigma(2)}<\dots < x_{\sigma(n)}$, and denote by $dy_i$ the $1$-form corresponding to the last coordinate of the point $x_i$. Now consider the form $dl_1\wedge dl_2 \wedge dy_{\sigma(1)} \wedge dy_{\sigma(2)} \wedge \ldots \wedge dy_{\sigma(n)}$. It induces an orientation on each $(n+d)$-dimensional cell, and similarly $dl \wedge dy_{\sigma(1)} \wedge  dy_{\sigma(2)} \wedge \ldots \wedge d y_{\sigma(n)}$ induces an orientation on each $(n+d-1)$-dimensional cell at the boundary of the previous one. The orientations of the cells were defined in terms of differential forms. The sign of the attaching map depends only on the lexicographical order, it is the sign of $\sigma$ seen as a permutation. 
Since the orientations only depend on the sign of $\sigma$, when we
project to $F_n(\Reals^d)/A_n$ the orientations of the cells
$\bold{C_i}(F_n'(\Reals^d),pt)$ induce well defined orientations of
the cells $\bold{C_i}(F_n'(\Reals^d)/A_n,pt)$. To compare the sign of
the attaching map at two different fibers, consider a path between
them in the labeled configuration space. This path projects to a loop
in the unlabeled configuration space and it projects down to a path in
its double cover $F_n(\Reals^d)/A_n$. This path in $F_n(\Reals^d)/A_n$
is a loop if and only if the sign of the permutation is positive. If
it is a loop, then both fibers are attached with the same
orientation. Otherwise, the sign of the permutation is negative and it
projects down to a path with distinct endpoints in
$F_n(\Reals^d)/A_n$. In terms of a given generator $e_{n_1}'\otimes 1
\in \bold{C_{n+d-1}}(F_n(\Reals^d)/A_n,pt)\otimes \ZZ$, this
permutation acts nontrivially on the decomposition sending $e_{n_1}'$
to its ``antipodal'' cell  $\tau(e_{n_1}')$, but
$\tau(e_{n_1}')\otimes 1=e_{n_1}'\otimes -1$ and so the negative signs, one coming from contracting the differential
form and one coming from the action, cancel each other, giving the desired 
formula \eqref{eq:coeff}.

To finish the proof, we want to apply property~(P2) of the Euler
class.  Some caution has to be taken as we are dealing with twisted
coefficients and the base space is not a manifold. Here we recall the main
idea of its proof to keep track of the coefficients, see \cite{bott}
for details. Let $i$ be the zero section and $s$ a
transversal section. By definition the Thom class $[u]$ is taken with
$or(\eta^{d-1})$ coefficients.  The zero section is Poincare dual to $[u]$ with respect to the pair $(E,E_0)$ with the orientation of the total space $or(E)$. The restriction of the bundle $E \to F_n(\Reals^d)/\Sigma_n$ to the intersection $Z=i(F_n(\Reals^d)/\Sigma_n) \cap s(F_n(\Reals^d)/\Sigma_n)$ is isomorphic to the normal bundle $\nu$ of the submanifold $s^{-1}(Z)=e$ inside $F_n(\Reals^d)/\Sigma_n$ (here transversality is used). We will denote a tubular neighborhood of this bundle inside $F_n(\Reals^d)/\Sigma_n$ by $t(\nu)$. We identify $F_n(\Reals^d)/\Sigma_n$ with $i(F_n(\Reals^d)/\Sigma_n)$ and $Z$ with $e$.  The Thom class of the bundle $E \to F_n(\Reals^d)/\Sigma_n$ restricts to the Thom class $E|_Z \to Z$, which in turn is mapped to $$[Z] \in H_{dim(Z)}(D(E|_Z),S(E|_Z); or(\nu)\otimes or(E|_Z))= H_{dim(Z)}(D(E|_Z),S(E|_Z); or(E))$$ under the Poincare duality with respect to the pair $(D(E|_Z),S(E|_Z))$, where $D$ and $S$ denote disc and sphere bundles respectively. Thinking of this pair included in the manifold $F_n(\Reals^d)/\Sigma_n$ via the vector bundle isomorphism, excision gives $$H_{dim(Z)}(D(E_Z),S(E_Z); or(E))=H_{dim(Z)}(F_n(\Reals^d)/\Sigma_n, F_n(\Reals^d)/\Sigma_n-t(\nu);or(E)).$$ On the other hand, by definition, the Euler class is the restriction of the Thom class $\chi:=i^*([u])$; we can factor this restriction through the Thom class of $\nu$. Naturality of the cap product implies that the inclusion of the cycle  $$[e]\in H_{n+d-1}(Z; or(Z)) \to H_{n+d-1}(F_n'(\Reals^d)/\Sigma,pt; or(E))$$ is the image of the Euler class
\begin{align*}
\chi(\eta^{d-1})& {} \in H^{(n-1)(d-1)}(F_n(\Reals^d)/\Sigma_n;or(E) \otimes or(F_n'(\Reals^d)/\Sigma_n))\\
& =H^{(n-1)(d-1)}(F_n(\Reals^d)/\Sigma_n;or(\eta^{d-1}))
\end{align*}
under the Lefschetz-Poincare morphism: $[a] \to [a]\cap [(F_n'(\Reals^d)/\Sigma,pt]$. We have seen that the relative
cycle $[e]$ is nontrivial in $H_{n+d-1}(F_n'(\Reals^d)/\Sigma,pt;
or(E))$ whenever $n$ is a prime power and $or(E)=\widehat{Z}$, so we
just need to check the latter. Now, $$or(E)=or(\eta^{d-1}) \otimes
or(F_n(\Reals^d)/\Sigma_n).$$ We can assume that $n$ is odd as we
already dealt with $n=2^k$. If $d$ is odd, then $or(\eta^{d-1})=\ZZ$
and $or(F_n(\Reals^d)/\Sigma_n))=\widehat{\ZZ}$. If $d$ is even, then
$or(\eta^{d-1})=\widehat{\ZZ}$ and
$or(F_n(\Reals^d)/\Sigma_n)=\ZZ$. In both cases $or(E)=\widehat{\ZZ}$.
\end{proof}

\begin{proof}[Proof of theorem~\ref{extra}]
The map 
\[
  r_{K,(\frac{1}{n},\ldots, \frac{1}{n})}:F_n(\Reals^d) \to \Reals^{n-1}
\]
given by theorem~\ref{AHA} is continuous by lemma~\ref{lemma2}. The map 
\[
  C \colon F_n(\Reals^d) \times \Reals^{n-1} \to (\mathcal{K}^d)^n
\]
taking a point in configuration space and a set of
radii to its power diagram, is continuous by lemma~\ref{lemma1}. The
functional $G$ is continuous by assumption. We can conclude that the
composition map
\[
  F_n(\Reals^d)\xrightarrow{Id \times r}
  F_n(\Reals^d) \times \Reals^{n-1} \xrightarrow{C} (\mathcal{K}^d)^n
  \xrightarrow{G\times G \times \dots \times G} \Reals^{n(d-1)}
\]
is continuous. By the choice $q=(\frac{1}{n},\ldots \frac{1}{n})$ it is
$\Sigma_n$-equivariant. By theorem~\ref{top} and the discussion surrounding it, this map has to intersect the $(d-1)$-dimensional diagonal.
\end{proof}

\subsection*{Remarks and questions}
Optimal transport makes sense in any measure metric space and it has been instrumental in proving inequalities in Riemmanian and Convex Geometry. Our proof extends to other measure metric manifolds $M$, where the geometry of the cells might not be convex anymore, but they retain some desirable properties related to the transport problem.  
Recall the Waist inequality of \cite{gro},
\begin{theorem} For any continuous map $f\colon \Sphere^d \to \Reals^k$, there exists $z \in \Reals^k$ such that,  $$vol(f^{-1}(z)+t) \geq vol(S^{d-k}+t),$$ for all $t$, where $\Sphere^{d-k}$ is an equatorial sphere inside $\Sphere^d$.
\end{theorem}

We remark that $t$ is not assumed to be small.

The proof of this Waist theorem relies on ideas from the celebrated \emph{localization} technique. The main new tool to perform a localization-type argument when the image is not one-dimensional is a topological result which is very similar to theorem~\ref{extra}. The second ingredient consists of comparison inequalities related to Brunn's inequality. Finally Rohlin's theory delivers the result.
Is it possible to extend the analytical content of Gromov's proof to obtain a \emph{comparison} Waist inequality for manifolds of positive curvature with the normalized Riemannian volume? The key missing ingredient is a Brunn type result regarding the cells of the optimal transport from an absolutely continuous measure to a discrete measure.

If we fix the measure, the convex set, and the continuous functionals, theorem ~\ref{extra} gives a power diagram and a discrete measure for each $n=p^k$. What can we say about the combinatorics of power diagrams for some natural continuous functionals? What can we say about measures in the limit when we let $n$ go to infinity? 

What if we drop the curvature constraint? Are there other variational problems on the space of cycles or related geometric inequalities that can be approached through variational or topological problems on configuration space?

Gromov's paper \cite{gro} contains several related questions that should be revisited with the tools and ideas of optimal transport at hand, see section 9.3.A in that paper. 

Meanwhile, the Nandakumar-Ramana Rao conjecture remains open for $6$ pieces in the plane and any other number that is not a prime power.

\subsection*{Acknowledgments}
The first author would like to thank Steven Simon and Sylvain Cappell for useful discussions. 
The second author would like to thank Daniel Klain and
Erwin Lutwak for a helpful pointer.  We are grateful to Imre
B\'ar\'any for introducing us to the NRR~conjecture, and for suggesting improvements on the
exposition. We are grateful to Roman Karasev for informing us about his result and for his friendly correspondence. We are  grateful to Ed Miller for enlightening discussions.
The first author acknowledges and thanks the support from CONACyT.
The research of the second author has been supported in part by a
grant No.~2006/194 from the U.S.-Israel Binational Science Foundation,
by NSA MSP Grant H98230-10-1-0210, and by NSF Grant CCF-08-30691.
Research partially supported by the
Centre Interfacultaire Bernoulli at EPFL, Lausanne.

\subsection{Appendix: Equipartitioning measures, second proof}
As mentioned earlier, Pablo Soberon found a clever modification of our
original argument to prove a generalization of the ham sandwich theorem. Here we give a sketch of that proof.
Choosing $n=p$ prime and restricting to a cyclic subgroup $\ZZ_p\subset \Sigma_p$, the action associated with the corresponding permutation representation is free on the sphere. The Borsuk-Ulam theorem for $\ZZ_p$ actions can derived from the cup product structure of Lens spaces or from the following theorem of Dold. See~\cite{mat} for an excellent exposition on $G$-spaces.
\begin{theorem}[Dold]\label{d}
  If $X$ is an $n$-connected space, $Y$ a space of dimension at
  most~$n$, and group $G$ acts on $X$ and acts freely on $Y$, then
  there is no $G$-equivariant map from $X$ to $Y$.
\end{theorem} 
\begin{proof}[Proof of corollary~\ref{sob} \textup{\cite{s}}]
Assume that all the measures are probability measures. Let $p$ be a
prime that divides $n$. This time configuration space will be
substituted by $(\Reals^d)^p\setminus \Delta_d$, where $\Delta_d$ is
the $d$-dimensional diagonal, that is, we are looking at $p$-tuples of
points $\langle x_1,x_2, \ldots x_p\rangle$ such that not all of them
coincide. For each ``configuration'' in this space and each measure
we apply theorem~\ref{AHA} and obtain a power diagram and a
corresponding transformation for each measure. We are solving $d$
transport problems. Each transformation corresponds to a set of $p-1$ real values, the radii of the power diagram associated to that transformation. This is the same map we defined before (lemma~\ref{lemma2}) except when two or more points collide. In that case, all the points that collide share a single cell and we assign to each of them the same radius (corresponding to that cell). Theorem~\ref{AHA} and the discussion surrounding it imply that there is a unique (up to scaling) set of radii, again we have continuity of the transport map with respect to the source measure (even when points collide) and the same argument as at the end of lemma~\ref{lemma2} implies that this extended map to the set of radii is continuous. Observe that this map is $\ZZ_p$-equivariant.
 We have a continuous $\ZZ_p$-equivariant map $$(\Reals^d)^p \to (\Reals^{p-1})^d,$$ and we want to show that it intersects $\Delta_{p-1}$, the $(p-1)$-dimensional diagonal corresponding to all power diagrams (transports) having the same set of radii.  
As before, assuming it does not intersect $\Delta_{p-1}$ we can deformation retract it to obtain a $\ZZ_p$-equivariant map $$\Sphere^{d(p-1)-1} \to \Sphere^{(d-1)(p-1)-1},$$ but this contradicts Borsuk-Ulam theorem for $\ZZ_p$ actions.  This means that we have a power diagram with $k$ cells, with $1<k \leq p$, such that each cell has the same fraction of each measure, that is, $$\mu_1(C_j)=\mu_2(C_j)=\ldots \mu_d(C_j)=\frac{l_j}{p},$$ for some $l_j$ with $1\leq l_j < p$. Now we subdivide each cell which has $l_j>1$ measure. Consider another prime $p'$ that divides $l_j$ and apply the same reasoning. This process can be repeated, and it stops when each subcell has measure $\frac{1}{p}$ for all measures. We explained how to subdivide into $p$ cells, for $p\mid n$. Going through the factorization of $n$, we further subdivide each cell as in Karasev's proof.
\end{proof}

\end{document}